\pdfoutput=1
\documentclass{amsart}

\usepackage[T1]{fontenc}
\usepackage[utf8]{inputenc}
\usepackage{amssymb,amsmath,mathtools}
\usepackage{alphalph}
\usepackage{enumitem}
\usepackage{tikz}
\usepackage{verbatim}
\usetikzlibrary{arrows.meta,calc,positioning}
\usepackage[hidelinks]{hyperref}
\usepackage[margin=1.4in]{geometry}
\usepackage{quiver}

\hypersetup{
	colorlinks   = true, %Colours links instead of ugly boxes
	urlcolor     = blue, %Colour for external hyperlinks
	linkcolor    = black, %Colour of internal links
	citecolor   = blue %Colour of citations
}

\newcommand{\N}{\mathbb{N}}
\newcommand{\kk}{\mathbb{K}}    % ground field

\newcommand{\bA}{\mathbf{A}}
\newcommand{\bH}{\mathbf{H}}
\newcommand{\bI}{\mathbf{I}}
\newcommand{\bX}{\mathbf{X}}

\newcommand{\calF}{\mathcal{F}}
\newcommand{\calZ}{\mathcal{Z}}

\newcommand{\PaRCD}{\mathsf{PaRCD}}
\newcommand{\hPaRCD}{{\PaRCD}_{\kk}}
\newcommand{\PaCD}{\mathsf{PaCD}}
\newcommand{\hPaCD}{{\PaCD}_{\kk}}
\newcommand{\magma}{\Omega}
\newcommand{\cyc}{\mathbf{cyc}}
\newcommand{\Op}{\mathbf{Op}}

\newcommand{\id}{\operatorname{id}}
\newcommand{\Hom}{\operatorname{Hom}}
\newcommand{\Aut}{\operatorname{Aut}}
\DeclareMathOperator{\Prim}{Prim}
\DeclareMathOperator{\spanlin}{span}

\newcommand{\GRT}{{\mathsf{GRT}}}
\newcommand{\ASTR}{{\mathsf{ASTR}}}
\newcommand{\lgrt}{{\mathfrak{grt}}}
\newcommand{\ff}{\mathfrak{f}}
\newcommand{\ft}{\mathfrak{t}}
\newcommand{\fft}{\mathfrak{ft}}
\newcommand{\fg}{\mathfrak{g}}

\newcommand{\hff}{\widehat{\ff}_2}          % completed free Lie algebra on x,y
\newcommand{\fd}{\mathfrak{d}}
\newcommand{\fh}{\mathfrak{h}}
\newcommand{\fk}{\mathfrak{k}}
\newcommand{\UU}{\widehat U}
\newcommand{\cl}[1]{\overline{#1}}

\newtheorem{theorem}{Theorem}[section]
\newtheorem{thm}[theorem]{Theorem}
\newtheorem{prop}[theorem]{Proposition}
\newtheorem{lemma}[theorem]{Lemma}
\newtheorem{cor}[theorem]{Corollary}

\newtheoremstyle{alphatheorem}
  {}{}{\itshape}{}{\bfseries}{.}{ }
  {\thmname{#1} \thmnumber{\AlphAlph{\value{introtheorem}}}\thmnote{ (#3)}}
\theoremstyle{alphatheorem}
\newtheorem{introtheorem}{Theorem}

\theoremstyle{definition}
\newtheorem{definition}[theorem]{Definition}

\newtheorem{remark}[theorem]{Remark}

\numberwithin{equation}{section}

\begin{document}

\title[Cyclic reflections and the flip equation]{Cyclic reflections and the flip equation for the graded Grothendieck--Teichmüller group}

%    author information
\author{Chandan Singh}
\address{School of Mathematics and Statistics, The University of Melbourne, Melbourne, Victoria, Australia}
\email{chandan.singh1@unimelb.edu.au}

%    \subjclass is required.


\date{}

%    "Communicated by" -- provide editor's name; required.
\commby{}

%    Abstract is required.
\begin{abstract}
Furusho, Hirose and Komiyama reformulated the pentagon equation using an involution called the flip, which led to a new characterisation of the graded Grothendieck--Teichm\"uller group $\mathsf{GRT}_1$. We show that the flip is a reflection in the cyclic structure of the operad of framed Drinfeld–Kohno Lie algebras, with the framings removed. Applying the reflection twice to the pentagon gives the two-cycle relation for every Drinfeld associator. This relation is precisely the condition for cyclic compatibility on the associativity morphism, giving a new proof that $\mathsf{GRT}_1$ acts by cyclic automorphisms of parenthesized ribbon chord diagrams.
\end{abstract}

\maketitle

\section{Introduction}\label{sec:intro}

The graded Grothendieck--Teichm\"uller group $\GRT_1$ of Drinfeld \cite{Drin} acts freely and transitively on the set of Drinfeld associators with a fixed non-zero coupling constant. It is the group of group-like series $g(x,y)$ in two non-commuting variables satisfying the pentagon, two-cycle and three-cycle relations, and by Furusho \cite{Furusho_pentagon} the pentagon alone suffices once the quadratic part of $g$ vanishes. 
\medskip 

$\GRT_1$ has two descriptions that look unrelated. On the operadic side, $\GRT_1$ is the group of automorphisms of the prounipotent completed operad of parenthesized chord diagrams \cite{BN}, \cite[Chapter~10]{FresseBook1}. These automorphisms extend to the framed operad, where they preserve the cyclic structure \cite{willwacher2024cyclic,singh2026cyclic}. On the mould-theoretic side, Furusho, Hirose and Komiyama \cite{FHK} recast the pentagon equation using an involution of the Drinfeld--Kohno Lie algebras, the \emph{flip}, and use it to identify associators with Écalle's balanced moulds \cite{Jean2003}. This paper explains how the two descriptions meet: the flip is a reflection in the cyclic structure of chord diagrams.

% Furusho, Hirose and Komiyama \cite{FHK} recast of the pentagon for all associators using an involution, the \textit{flip}. We explain this reformulation through the cyclic structure of chord diagrams and use it to give a new proof that $\GRT_1$ acts by cyclic automorphisms.

\subsection*{The flip}
Let $\ft_n$ be the Drinfeld--Kohno Lie algebra on the strands $0,\dots,n-1$, with generators $t_{ij}=t_{ji}$. The \emph{flip} of \cite{FHK} is the involution of $\ft_n$ given by
\begin{equation}\label{eq:intro-flip}
F_n(t_{ij})=t_{n-i,\,n-j}\quad(0<i,j<n),\quad F_n(t_{0j})=-\sum_{k<n}t_{k,\,n-j}\quad(0<j<n),
\end{equation}
with $t_{kk}=0$. It is used on the \emph{dagger subalgebra} $\UU(\ft_n)^\dagger=\bigcap_{0<j<n}\ker\partial_{0j}$ of the completed enveloping algebra, where $\partial_{0j}$ is the derivation extracting the coefficient of $t_{0j}$ (Section~\ref{subsec:flip}). By \cite[Theorem~78]{FHK}, a group-like series with no linear term in $x$ is an associator up to a factor $\exp(\alpha y)$ if and only if it satisfies a \emph{flip equation}: $F_4$ carries one path around the pentagon to the other, up to a correction supported on the strands $1,2,3$.

\medskip 

Three aspects of \eqref{eq:intro-flip} require further explanation. First, the flip is not induced by a permutation of the strands $0,\dots,n-1$, since $F_n(t_{0j})$ is a sum of generators. Second, its domain is cut out by derivations with no evident operadic meaning. Third, the hypothesis on the linear term in \cite[Theorem~78]{FHK} has no clear origin. Modulo the centre $\kk c_n$, $F_n$ is Ihara's action of $i\mapsto n-i$ on the sphere braid Lie algebra $\ft_n/\kk c_n$ \cite[Section~5.3]{ihara_braids}, which suggests a cyclic explanation. However, $(\kk c_n)_n$ is not an operadic ideal (Remark~\ref{rem:noquotient}), so this does not relate the flip to 
the operadic description of $\GRT_1$. We keep the framings instead. In the framed operad the reflection is defined exactly, and all three questions have the same answer.

\subsection*{The flip is a cyclic reflection}
The operad of parenthesized chord diagrams has a framed version $\hPaRCD$, with coefficients in the framed Drinfeld--Kohno Lie algebras $\fft_n\supset\ft_n$. This is a \emph{cyclic} operad \cite{severa2009formality,campos2019configuration,willwacher2024cyclic}: the group $\Sigma_{n+1}$ of permutations of all $n+1$ legs of an arity-$n$ operation acts, exchanging inputs with the output. 
Our first result is that the flip is the reflection $\rho_n\colon i\mapsto n-i$ of the $n+1$ legs, where the extra leg is the output. The reflection produces framings $T_{ii}$, which we compute exactly (Theorem~\ref{thmA}). The flip is what remains when they are discarded, and the dagger subalgebra is precisely the locus where no framing appears. For the pentagon, the same mechanism explains the hypothesis of \cite[Theorem~78]{FHK}: the reflected pentagon of a group-like series is free of framings exactly when the series has no linear term in $x$. 

In arity four, $\rho_4$ is a mirror symmetry of Mac Lane's pentagon, and the flip equation says that a path and its mirror image differ by the edge fixed by the mirror (Figure~\ref{fig:pentagon}). Applying the mirror twice yields the two-cycle relation $g(x,y) g(y,x) =1$ for every associator. Moreover, the two-cycle relation is exactly the condition for the automorphism of $\hPaRCD$ determined by $g$ to commute with the cyclic action on the associativity morphism (Theorem~\ref{thmB}). Hence the pentagon alone forces cyclic compatibility on the associativity morphism, and every automorphism in the classification of Bar-Natan \cite{BN} and Fresse \cite{FresseBook1} is cyclic.

\subsection*{Statements}
Let $s\colon\ft_n\hookrightarrow\fft_n$ and $q\colon\fft_n\twoheadrightarrow\ft_n$ be the inclusion and the projection killing the framings $T_{ii}$. Labelling the output of an arity-$n$ operation by $n$, write $\sigma\mapsto P_\sigma$ for the action of $\Sigma_{n+1}$ on $\fft_n$, and put $\rho_n(i)=n-i$ and $\tau_n=(0\,n)$. Let $\fd_n=\cl{[\ft_n,\ft_n]}$ and $\fh_n=\fd_n\oplus\spanlin\{t_{ij}:0<i<j<n\}$. 

\begin{introtheorem}[Cyclic origin of the flip]\label{thmA}
Let $n\geq2$.
\begin{enumerate}[label=\textup{(\arabic*)}]
\item $F_n=q\,P_{\rho_n}\,s$, and $P_{\rho_n}s=sF_n$ on $\UU(\fd_n)$.
\item For every $u\in\UU(\ft_n)$,
\[
P_{\rho_n}s(u)=\sum_{\nu\in\N^{n-1}}\frac{(-1)^{|\nu|}}{\nu!}\,T^\nu\,sF_n\bigl(\partial^\nu u\bigr),
\]
where  for $\nu\in\N^{n-1}$, $T^\nu=\prod_{j=1}^{n-1}T_{n-j,n-j}^{\nu_j}$, $\partial^\nu=\prod_{j=1}^{n-1}\partial_{0j}^{\nu_j}$, $|\nu|=\sum_j\nu_j$ and $\nu!=\prod_j\nu_j!$.

\item $\UU(\ft_n)^\dagger=\bigl\{u\in\UU(\ft_n):P_{\rho_n}s(u)\in s\,\UU(\ft_n)\bigr\}=\UU(\fh_n)$, and $P_{\rho_n}s=sF_n$ on $\UU(\ft_n)^\dagger$.
\item $\bigcap_{\sigma\in\Sigma_{n+1}}P_\sigma\bigl(s\,\UU(\ft_n)\bigr)=s\,\UU(\fd_n)$, and $s\fd=[\fft,\fft]$ is the largest cyclic suboperad of $\fft$ contained in $s\ft$.
\end{enumerate}
\end{introtheorem}

% In particular, the involutivity of $F_n$ \cite[Lemma~60]{FHK} is the relation $\rho_n^2=1$, and the dagger subalgebra \cite[Definition~59, Lemma~61]{FHK} is the enveloping algebra of the Lie ideal $\fh_n$.

% For $g\in\exp(\hff')$, where $\hff'$ is the derived algebra of the completed free Lie algebra on $x,y$, put

Let $\hff'$ be the derived algebra of the completed free Lie algebra $\hff$ on $x,y$. For $g\in\exp(\hff)$ put $S_g=g^{0,1,2}g^{0,12,3}$ and $S'_g=g^{01,2,3}g^{0,1,23}$ (see Section~\ref{subsec:grt}), so that the pentagon equation \textup{(P)} reads $S'_g=S_g\,g^{1,2,3}$.

% Let $\hff'$ be the derived algebra of the completed free Lie algebra $\hff$ on $x,y$. For $g\in\exp(\hff)$ put
% \[
% S_g=g^{0,1,2}g^{0,12,3},\qquad S'_g=g^{01,2,3}g^{0,1,23},\qquad C_g=S_g^{-1}F_4(S_g),
% \]
% let (P) denote the pentagon equation $S'_g=S_g\,g^{1,2,3}$, and let $\fd_{123}\subset\fd_4$ be the derived algebra of the copy of $\ft_3$ on the strands $1,2,3$ (see Section~\ref{subsec:grt} for the notation $g^{A,B,C}$). 

\medskip 

For $g\in\exp(\hff')$ let $\Theta_g$ be the $\Sigma_3$-equivariant assignment $\bA\mapsto g(T_{01},T_{12})\bA$, where $\bA$ is the associativity morphism of $\hPaRCD$. For $g\in\GRT_1$ this is the automorphism of Theorem~\ref{thm:classification}. Let $\Aut^+_{\cyc,1}(\hPaRCD)$ be the group of object-fixing automorphisms of the cyclic operad $\hPaRCD$ fixing the framing, the chord and the braiding and preserving the unframed suboperad $\hPaCD$.

\begin{introtheorem}[Flip equation and cyclic automorphisms]\label{thmB}
\leavevmode
\begin{enumerate}[label=\textup{(\arabic*)}]
\item Let $g\in\exp(\hff)$. Then $P_{\rho_4}s(S_g)\in s\,\UU(\ft_4)$ if and only if $g$ has no linear term in $x$, and in that case $P_{\rho_4}s(S_g)=s(S'_g)$. Consequently
\[
\bigl\{g\in\exp(\hff):P_{\rho_4}s(S_g)=s(S_g\,h^{1,2,3})\text{ for some }h\in\exp(\hff')\bigr\}=\ASTR\cdot\exp(\kk y),
\]
and $h=\varphi$ when $g=\varphi\exp(\alpha y)$ with $\varphi\in\ASTR$. This is the flip equation of \cite[Theorem~78]{FHK}, without the hypothesis on the linear term in $x$. For $g\in\ASTR$ the relation $P_{\rho_4}^2=1$ gives $g(x,y)\,g(y,x)=1$.

\item Let $g\in\exp(\hff')$. Then $\Theta_g$ commutes with the action of $\Sigma_4$ on $\bA$ if and only if $g(x,y)\,g(y,x)=1$.
\end{enumerate}
Consequently, together with Theorem~\ref{thm:classification}, $g\mapsto\Theta_g$ is a group isomorphism $(\GRT_1, \circledast) \cong\Aut^+_{\cyc,1}(\hPaRCD)$, for the product $\circledast$ defined in ~\eqref{eq:grouplaw}.
\end{introtheorem}

Theorem~\ref{thmB}(1) refines \cite[Theorem~78]{FHK} by the framed identity $P_{\rho_4}s(S_g)=s(S'_g)$, which identifies the flip equation with a symmetry of the pentagon in $\hPaRCD$ and removes the hypothesis on the linear term. This gives a short proof (using $P_{\rho_4}^2=1$) that the pentagon implies the two-cycle relation, an alternative to \cite[Lemma~6]{Furusho_pentagon}. The final isomorphism recovers, for $\GRT_1$, the cyclicity of the Grothendieck--Teichm\"uller action \cite[Proposition~5.2]{willwacher2024cyclic}, \cite[Theorem~3.5]{singh2026cyclic}, the latter after restricting to automorphisms preserving $\hPaCD$.

\medskip 
Theorem~\ref{thmB}(2) makes explicit the cyclic compatibility condition from \cite[Proposition 5.2]{willwacher2024cyclic}. 
We show it follows from applying the framed reflection twice to the pentagon and precisely characterise the subalgebras where the reflection stays unframed.

\medskip 

Section~\ref{sec:prelim} recalls the framework, Section~\ref{sec:A} proves Theorem~\ref{thmA}, Section~\ref{sec:arity4} proves Theorem~\ref{thmB}(1) and interprets it as a reflected pentagon, and Section~\ref{sec:cyclic} completes the proof of Theorem~\ref{thmB}(2).

\subsection*{Conventions} Throughout, $\kk$ is a field of characteristic zero, Lie algebras are completed for the weight grading, and $\calF^k$ is the part of weight at least $k$. Inputs of an arity-$n$ operation are labelled $0,\dots,n-1$ and the output $n$. The labels of \cite{singh2026cyclic,willwacher2024cyclic} are recovered by $i\mapsto i+1\pmod{n+1}$. Morphisms are composed from left to right.

\subsection*{Acknowledgements}The author thanks H. Furusho for suggesting the extension to $\ASTR$ and for his talk at the seminar \textit{Johnson homomorphisms and related topics} in July 2026,  which motivated the connection between the flip construction and the author's prior work on cyclic operads. The author also acknowledges the support of the SMRI–AustMS Lift-Off Fellowship.

\section{Preliminaries}\label{sec:prelim}

This section fixes notation and reviews the framed Drinfeld--Kohno operad $\fft$ and its cyclic structure, the group $\GRT_1$ and the flip of \cite{FHK}, and the cyclic operad $\hPaRCD$ with its automorphisms. 

\subsection{The framed Drinfeld--Kohno operad and its cyclic structure}\label{subsec:ft}

\begin{definition}\label{def:ft}
For $n\geq1$, the \emph{framed Drinfeld--Kohno Lie algebra} $\fft_n$ is the completed Lie algebra generated by $T_{ij}=T_{ji}$, $0\leq i,j\leq n-1$, subject to
\begin{equation}\label{eq:ft-relations}
\begin{aligned}
[T_{ij},T_{kl}]&=0 &&\text{for }\{i,j\}\cap\{k,l\}=\varnothing,\\
[T_{ij},T_{ki}+T_{kj}]&=0 &&\text{for distinct }i,j,k,\\
[T_{ii},T_{jk}]&=0 &&\text{for all }i,j,k.
\end{aligned}
\end{equation}
The \emph{Drinfeld--Kohno Lie algebra} $\ft_n$ is generated by $t_{ij}=t_{ji}$, $0\le i\ne j\le n-1$, subject to the first two families of relations. We put
\[
c_n=\sum_{0\le i<j<n}t_{ij}\in\ft_n,\qquad \fd_n=\cl{[\ft_n,\ft_n]}=\calF^2\ft_n .
\]
\end{definition}

The element $c_n$ is central in $\ft_n$. Let $s\colon\ft_n\hookrightarrow\fft_n$, $s(t_{ij})=T_{ij}$, and $q\colon\fft_n\twoheadrightarrow\ft_n$, $q(T_{ij})=t_{ij}$ for $i\ne j$ and $q(T_{ii})=0$. The framings $T_{ii}$ occur in \eqref{eq:ft-relations} only through $[T_{ii},T_{kl}] = 0 = [T_{ij},T_{kk}]$ and the third family of relations, so
\begin{equation}\label{eq:framing-splitting}
\fft_n=s(\ft_n)\oplus\calZ_n,\qquad \calZ_n=\bigoplus_{0\le i<n}\kk T_{ii},
\end{equation}
is a decomposition of Lie algebras with $\calZ_n$ central, and $q$ is the projection onto the first summand. Correspondingly, $\UU(\fft_n)=s\,\UU(\ft_n)\,\widehat\otimes\,\kk[[T_{00},\dots,T_{n-1,n-1}]]$.

The Lie algebras $\fft=(\fft_n)_{n\ge1}$ form an operad in complete Lie algebras, with monoidal product $\oplus$ \cite{severa2009formality}, \cite[Section 1]{singh2026cyclic}, \cite[Section 5.1]{willwacher2024cyclic}. The partial composition $\circ_a\colon\fft_m\oplus\fft_k\to\fft_{m+k-1}$, $0\le a<m$, replaces input $a$ by the block $B=\{a,\dots,a+k-1\}$. It is given by $0\circ_aT_{ij}=T_{i+a,j+a}$ for $0\le i,j<k$ and, for $0\le i<j<m$,
\begin{equation}\label{eq:cabling}
\begin{gathered}
T_{ij}\circ_a0=
\begin{cases}
T_{i+k-1,\,j+k-1}&a<i,\\
\sum_{p\in B}T_{p,\,j+k-1}&a=i,\\
T_{i,\,j+k-1}&i<a<j,\\
\sum_{p\in B}T_{ip}&j=a,\\
T_{ij}&j<a,
\end{cases}
\\[1ex]
T_{ii}\circ_a0=
\begin{cases}
T_{i+k-1,\,i+k-1}&a<i,\\
\sum_{p\in B}T_{pp}+2\sum_{p<p'\in B}T_{pp'}&a=i,\\
T_{ii}&i<a.
\end{cases}
\end{gathered}
\end{equation}
The same formulas without framings make $\ft=(\ft_n)_{n\ge1}$ an operad, and $s\colon\ft\to\fft$ is a morphism of operads. Applying $\UU$ levelwise yields operads $\UU(\ft)$ and $\UU(\fft)$ in complete Hopf algebras.

\begin{remark}[Spherical presentation]\label{rem:spherical}
Put
\begin{equation}\label{eq:output}
T_{ni}=T_{in}=-\sum_{j<n}T_{ji}\quad(0\le i<n),\qquad T_{nn}=\sum_{i,j<n}T_{ij}=\sum_{i<n}T_{ii}+2\,s(c_n).
\end{equation}
Then $\fft_n$ is presented by generators $T_{ij}=T_{ji}$, $0\le i,j\le n$, subject to $[T_{ij},T_{kl}]=0$ for $\{i,j\}\cap\{k,l\}=\varnothing$, centrality of every $T_{ii}$, and $\sum_{j=0}^nT_{ij}=0$ for $0\le i\le n$ \cite[\S5.1]{willwacher2024cyclic} (compare \cite[Section5.3]{ihara_braids} and \cite[Definition~3.3]{singh2026cyclic}).
\end{remark}

The group $\Sigma_{n+1}$ of permutations of $\{0,\dots,n\}$ acts on $\fft_n$ by $P_\sigma(T_{ij})=T_{\sigma(i)\sigma(j)}$ in the presentation of Remark~\ref{rem:spherical}, and hence on $\UU(\fft_n)$ and $\exp(\fft_n)$. This makes $\fft$ a cyclic operad in complete Lie algebras \cite{severa2009formality,campos2019configuration}, \cite[\S5.1]{willwacher2024cyclic}; on $\Sigma_n\subset\Sigma_{n+1}$ it is the usual action by permutation of inputs. We use two elements of $\Sigma_{n+1}$: the transposition $\tau_n=(0\,n)$ and the reflection $\rho_n(i)=n-i$. By \eqref{eq:output}, for $0\le i,j<n$,
\begin{equation}\label{eq:tau}
\begin{gathered}
P_{\tau_n}(T_{ij})=
\begin{cases}
T_{ij}& i,j\ne0,\\
-\sum_{k<n}T_{kj}& i=0\ne j,\\
\sum_{k,l<n}T_{kl}& i=j=0,
\end{cases}
\\[1ex]
P_{\rho_n}(T_{ij})=
\begin{cases}
T_{n-i,\,n-j}& 0<i,j<n,\\
-\sum_{k<n}T_{k,\,n-j}& i=0<j,\\
\sum_{k,l<n}T_{kl}& i=j=0.
\end{cases}
\end{gathered}
\end{equation}
Since $\Sigma_n$ and $\tau_n$ generate $\Sigma_{n+1}$, the cyclic structure is determined by the operators $P_{\tau_n}$. Under the relabelling $i\mapsto i+1\pmod{n+1}$, $\tau_n$ becomes the transposition exchanging the output $0$ with the input $1$, which is the operator denoted $z^*$ in \cite[Sections~1--3]{singh2026cyclic} and $\tau=(01)$ in \cite[\S5]{willwacher2024cyclic}.

\subsection{The graded Grothendieck--Teichm\"uller group}\label{subsec:grt}

Let $\hff=\hff(x,y)$ be the completed free Lie algebra on $x$ and $y$, so that $\UU(\hff)=\kk\langle\!\langle x,y\rangle\!\rangle$, and let $\hff'=\cl{[\hff,\hff]}=\calF^2\hff$. For $\psi\in\UU(\hff)$ and elements $u,v$ of the augmentation ideal of a complete algebra, $\psi(u,v)$ denotes the image of $\psi$ under the continuous algebra map $x\mapsto u$, $y\mapsto v$. For pairwise disjoint sets $A,B,C$ of strands put $t_{AB}=\sum_{a\in A,b\in B}t_{ab}$ and $\psi^{A,B,C}=\psi(t_{AB},t_{BC})$ \cite[(1.51)]{FHK}. In $\ft_4$ we use
\begin{equation}\label{eq:five}
\begin{aligned}
\psi^{0,1,2}&=\psi(t_{01},t_{12}),&\psi^{0,12,3}&=\psi(t_{01}+t_{02},\,t_{13}+t_{23}),&\psi^{1,2,3}&=\psi(t_{12},t_{23}),\\
\psi^{01,2,3}&=\psi(t_{02}+t_{12},\,t_{23}),&\psi^{0,1,23}&=\psi(t_{01},\,t_{12}+t_{13}),
\end{aligned}
\end{equation}
and write
\[
\begin{gathered}
S_g=g^{0,1,2}g^{0,12,3},\quad S'_g=g^{01,2,3}g^{0,1,23},\\
S_\psi=\psi^{0,1,2}+\psi^{0,12,3},\quad S'_\psi=\psi^{01,2,3}+\psi^{0,1,23}.
\end{gathered}
\]
The \emph{pentagon equation} for a group-like $g\in\exp(\hff)$ and its linearisation for $\psi\in\hff$ are
\begin{equation}\label{eq:P}
\text{(P)}\quad S'_g=S_g\,g^{1,2,3}\ \text{ in }\UU(\ft_4),\qquad\qquad
\text{(P$'$)}\quad S'_\psi=S_\psi+\psi^{1,2,3}\ \text{ in }\ft_4 .
\end{equation}

\begin{definition}\label{def:grt}
The \emph{graded Grothendieck--Teichm\"uller group} and its Lie algebra are
\[
\begin{gathered}
\GRT_1=\bigl\{g\in\exp(\calF^3\hff):\ g\text{ satisfies (P)}\bigr\},\\
\lgrt_1=\bigl\{\psi\in\calF^3\hff:\ \psi\text{ satisfies (P$'$)}\bigr\}.
\end{gathered}
\]
\end{definition}

\begin{remark}\label{rem:grt-drinfeld}
A group-like $g$ satisfying (P) has no linear terms \cite[Lemma 75(1)]{FHK}, so Definition~\ref{def:grt} agrees with \cite[Definition~73(3)]{FHK}. The map $g\mapsto g^{-1}$ carries \textup{(P)} to Drinfeld's pentagon as written in \cite{Drin,Furusho_pentagon}, and by \cite[Lemmas~6 and~9, Theorem~3]{Furusho_pentagon} it identifies $\GRT_1$ with Drinfeld's group \cite[\S5]{Drin}. Likewise $\lgrt_1$ is Drinfeld's Lie algebra. 
% $\GRT_1$ is a group under $g\circledast h=g\,h(g^{-1}xg,y)$ \cite[(2.2)]{FHK}.
\end{remark}

Following \cite[Definition~73]{FHK}, let $\ASTR$ be the set of $\varphi\in\exp(\hff)$ satisfying \textup{(P)}, including those with vanishing quadratic part. By \cite[Theorem~1]{Furusho_pentagon}, $\ASTR$ is the set of Drinfeld associators. Moreover $\ASTR\subset\exp(\hff')$ \cite[Lemma~75(1)]{FHK} and $\GRT_1=\ASTR\cap\exp(\calF^3\hff)$.

% \begin{definition}[{\cite[Definition~73, Lemma~75(1)]{FHK}}]
% let $\ASTR$ be the set of $\varphi\in\exp(\hff)$ satisfying \textup{(P)}, the set of Drinfeld associators. Then $\ASTR\subset\exp(\hff')$, and $\GRT_1=\ASTR\cap\exp(\calF^3\hff)$.   
% \end{definition}

\begin{remark}[Normalisation]\label{rem:normalisation}
The restriction to $\calF^3$ in Definition~\ref{def:grt} cannot be dropped. For $\psi=[x,y]$ one computes in $\ft_4$
\[
S'_\psi-S_\psi-\psi^{1,2,3}=-[t_{01},t_{23}]-[t_{02},t_{13}]=0,
\]
so $[x,y]$ satisfies (P$'$), whereas $[x,y]+[y,z]+[z,x]=3[x,y]$ for $z=-x-y$, so it violates the linearised three-cycle relation; compare \cite[\S1]{Furusho_pentagon}. Since (P$'$) is homogeneous for the weight grading and $\calF^2\hff/\calF^3\hff=\kk \cdot [x,y]$, it follows that
\[
\bigl\{\psi\in\hff':\ \psi\text{ satisfies (P$'$)}\bigr\}=\kk\cdot [x,y]\oplus\lgrt_1 .
\]
In particular, the set defined in \cite[Definition~73(4)]{FHK} by the Lie-likeness condition and the linearised pentagon equation alone is $\kk\cdot [x,y]\oplus\lgrt_1$, and not $\lgrt_1$.
\end{remark}

\subsection{The dagger subalgebra and the flip}\label{subsec:flip}

For $0<j<n$, let $e_{0j}\colon\UU(\ft_n)\to\kk$ be the continuous linear form extracting the coefficient of $t_{0j}$ in the weight-one component, and put $\partial_{0j}=(e_{0j}\otimes\id)\circ\Delta$. Since $e_{0j}(uv)=e_{0j}(u)\varepsilon(v)+\varepsilon(u)e_{0j}(v)$, each $\partial_{0j}$ is a continuous derivation, with $\partial_{0j}(t_{ab})=\delta_{\{a,b\},\{0,j\}}$, and the $\partial_{0j}$ commute because $\Delta$ is cocommutative. The \emph{dagger subalgebra} is
\[
\UU(\ft_n)^\dagger=\bigcap_{j=1}^{n-1}\ker\partial_{0j},
\]
a complete Hopf subalgebra of $\UU(\ft_n)$ \cite[Definition~59, Lemma~61]{FHK}. The \emph{flip} is the involution given on generators by
\begin{equation}\label{eq:flip}
F_n(t_{ij})=t_{n-i,\,n-j}\quad(0<i,j<n),\quad
F_n(t_{0j})=-\sum_{k<n}t_{k,\,n-j}\quad(0<j<n),
\end{equation}
with the convention $t_{kk}=0$ \cite[Lemma~60]{FHK}. In \cite{FHK} it is considered on $\UU(\ft_n)^\dagger$. The formulas \eqref{eq:flip} define a continuous automorphism of $\ft_n$, and hence of $\UU(\ft_n)$, preserving $\UU(\ft_n)^\dagger$ (Proposition~\ref{prop:reflection} and Corollary~\ref{cor:dagger}).

% By \cite[Theorem~78]{FHK}, a series $g\in\exp(\hff)$ with vanishing coefficient of $x$ satisfies $F_4(S_g)=S_g\,h^{1,2,3}$ for some $h\in\exp(\hff')$ if and only if $g=\varphi\exp(\alpha y)$ with $\varphi$ satisfying (P) and $\alpha\in\kk$. In the Lie case the corresponding solution space is $\kk y\oplus\kk[x,y]\oplus\lgrt_1$, by Remark~\ref{rem:normalisation}.

\subsection{Parenthesized ribbon chord diagrams}\label{subsec:parcd}

The \emph{magma operad} $\magma$ is the free symmetric operad in sets on one binary operation with free $\Sigma_2$-action. An element of $\magma(n)$ is a planar rooted binary tree with leaves labelled bijectively by $0,\dots,n-1$ and root labelled $n$, written as a parenthesized word such as $((01)2)$. Forgetting the root, an element of $\magma(n)$ is a trivalent tree with $n+1$ labelled flags and a cyclic order of the half-edges at each vertex. The group $\Sigma_{n+1}$ acts on $\magma(n)$ by relabelling the flags, $i\mapsto\sigma(i)$, keeping the cyclic orders at the vertices. We then re-root at the flag labelled $n$ and read the leaves in the cyclic order following the root. This extends the action of $\Sigma_n$ and makes $\magma$ a cyclic operad in sets.

\begin{definition}\label{def:parcd}
For $n\ge1$, let $\hPaRCD(n)$ be the groupoid with object set $\magma(n)$, with $\Hom(p,q)=\exp(\fft_n)$ for all $p,q$, and with composition given by multiplication: the composite of $u\colon p\to q$ followed by $v\colon q\to r$ is $uv\colon p\to r$. We write a morphism as $(u\colon p\to q)$ and call $u$ its \emph{coefficient}. The operad structure is that of $\magma$ on objects and that of $\UU(\fft)$ on coefficients, and $\sigma\in\Sigma_{n+1}$ acts by $\sigma$ on objects and by $P_\sigma$ on coefficients. This is the cyclic operad of \emph{parenthesized ribbon chord diagrams} \cite[\S1]{singh2026cyclic}, \cite[\S5.1]{willwacher2024cyclic}.
\end{definition}

For $p\in\magma(n)$, the automorphism group of $p$ is the prounipotent group $\exp(\fft_n)$, with Lie algebra $\fft_n$. We call the elements
\[
\bI=T_{00}\ \text{at }(0),\qquad \bH=T_{01}\ \text{at }(01)
\]
of these Lie algebras the \emph{framing} and the \emph{chord}, and consider the \emph{braiding} and \emph{associativity} morphisms
\[
\bX=(1\colon(01)\to(10)),\qquad \bA=(1\colon((01)2)\to(0(12))).
\]
Every morphism of $\hPaRCD$ is obtained from $\exp(\kk\bI)$, $\exp(\kk\bH)$, $\bX$ and $\bA$ by categorical and operadic composition, inverses and completion \cite{BN}, \cite[\S10.2]{FresseBook1}, \cite[\S1]{singh2026cyclic}. For $\sigma\in\Sigma_3$ with $(a,b,c)=(\sigma(0),\sigma(1),\sigma(2))$ we write $A_{abc}=P_\sigma\bA=(1\colon((ab)c)\to(a(bc)))$, so that $\bA=A_{012}$. Since hom-sets are constant, a morphism is determined by its source, its target and its coefficient. In particular $(1\colon p\to q)^{-1}=(1\colon q\to p)$.

\begin{lemma}\label{lem:tau-generators}
The transpositions $\tau_n$ act on the generators by
\[
P_{\tau_1}\bI=\bI,\qquad P_{\tau_2}\bH=\bigl(-T_{01}-T_{11}\ \text{at }(10)\bigr),\qquad P_{\tau_2}\bX=\bX^{-1},\qquad P_{\tau_3}\bA=A_{120}^{-1}.
\]
\end{lemma}
This is the computation of \cite[\S1]{singh2026cyclic} and \cite[Proposition~5.2]{willwacher2024cyclic} in our labels: on coefficients use \eqref{eq:tau}, and on objects relabel the flags and re-root at the output.

\subsection{Automorphisms of \texorpdfstring{$\hPaRCD$}{PaRCD}}\label{subsec:aut}

Let $\Aut^+_{\Op}(\hPaRCD)$ be the group of automorphisms of the operad $\hPaRCD$ which are the identity on
objects. All automorphisms are taken in prounipotent groupoids over $\kk$, so that their induced Lie algebra maps are continuous and $\kk$-linear. Each $\Theta \in \Aut^+_{\Op}(\hPaRCD)$ induces, for every $p\in\Omega(n)$, a continuous automorphism of $\Aut(p)=\exp(\fft_n)$ and hence a continuous $\kk$-linear automorphism of its Lie algebra $\fft_n$. 
% We say that $\Theta$ \emph{fixes the framing and the chord} if $\Theta(\bI)=\bI$ and $\Theta(\bH)=\bH$ in these Lie algebras.

The inclusions $s: \ft_n \rightarrow \fft_n$ defined by $t_{ij} \mapsto T_{ij}$, identify the \textit{operad of parenthesised chord diagrams} $\hPaCD$ with a suboperad of $\hPaRCD$, having the same objects and unframed morphisms. There is a well-known isomorphism $\GRT_1 \cong \Aut^+_{\Op,1}(\hPaCD)$ (\cite{BN}, \cite[Section 10.3]{FresseBook1}). Here $\Aut^+_{\Op,1}(\hPaCD)$ denotes the object-fixing automorphisms fixing $\bH$ and $\bX$.
% Let $\Aut^+_{\cyc,1}(\hPaRCD)$ be the subgroup of those $\Theta\in\Aut^+_{\Op,1}(\hPaRCD)$ with
% $\Theta\circ P_\sigma=P_\sigma\circ\Theta$ on $\hPaRCD(n)$ for all $n\ge1$ and $\sigma\in\Sigma_{n+1}$.

We write 
\[
\begin{aligned}
\Aut^+_{\Op,1}(\hPaRCD)
:=\{\Theta\in\Aut^+_{\Op}(\hPaRCD):\;&
\Theta(\exp(s\bI))=\exp(s\bI),\,
\Theta(\exp(s\bH))=\exp(s\bH), \\ \quad
\Theta(\bX)=\bX,
&\Theta(\hPaCD)=\hPaCD, \, \text{for all} \, s \in \kk\},
\end{aligned}
\] 
and
\begin{multline*}
\Aut^+_{\cyc,1}(\hPaRCD) :=\bigl\{\Theta\in\Aut^+_{\Op,1}(\hPaRCD):
\Theta\circ P_\sigma=P_\sigma\circ\Theta\ \text{for all }\sigma\in\Sigma_{n+1},\ n\ge1\bigr\}.
\end{multline*}

\begin{thm}\label{thm:classification}
For every $g\in\GRT_1$ there is a unique $\Theta_g\in\Aut^+_{\Op,1}(\hPaRCD)$ with $\Theta_g(\bA)=g(T_{01},T_{12})\,\bA$, and $g\mapsto\Theta_g$ is a bijection $\GRT_1\to\Aut^+_{\Op,1}(\hPaRCD)$.
\end{thm}

\begin{proof}
Restriction to $\hPaCD$ is injective, as $\hPaRCD$ is generated by $\exp(\kk\bI)$, $\exp(\kk\bH)$, $\bX$ and $\bA$. By \cite{BN}, \cite[Section 10.3]{FresseBook1}, the restriction is the automorphism $G_g$ of $\hPaCD$ with $G_g(\bX)=\bX$ and $G_g(\bA)=g(T_{01},T_{12})\bA$, for a unique $g\in\GRT_1$. Conversely, $G_g$ fixes $\exp(\lambda s(c_k))$ at every
$p\in\magma(k)$. Since $c_k$ is central, it suffices to take $p=(p'(k-1))$, where 
$\exp(\lambda s(c_k))=\exp(\lambda\bH)\circ_0\exp(\lambda s(c_{k-1}))$ by \eqref{eq:cabling}, so the claim follows by induction on $k$.
Using \eqref{eq:framing-splitting}, extend $G_g$ by the identity on $\mathcal{Z}_n$. By \eqref{eq:cabling},
\[
\big(s(v)e^{z}\big)\circ_a\big(s(w)e^{z'}\big)
=\big(s(v)\circ_a s(w)\big)\big(0\circ_a e^{2\mu_a s(c_k)}\big)e^{z''},
\]
where $\mu_a$ is the coefficient of $T_{aa}$ in $z$ and $z''\in Z$. Hence, the extension
is an operad automorphism preserving $\hPaCD$, and it is $\Theta_g$.
\end{proof}

Theorem~\ref{thm:classification} is a framed version of \cite{BN}, \cite[Section 10.3]{FresseBook1}. Allowing $\bI$ and $\bH$ to be rescaled by a common $\lambda \in \kk^{\times}$ gives $\GRT_1\rtimes\kk^\times$.

% Without the condition on $\bH$ one obtains $\GRT_1\rtimes\kk^\times$, with $\kk^\times$ rescaling chords.

The five bracketings of the word $0123$ are
\begin{equation}\label{eq:vertices}
\begin{gathered}
p_0=(((01)2)3),\quad p_1=((0(12))3),\quad p_2=(0((12)3)),\\
p_3=(0(1(23))),\quad p_4=((01)(23)).
\end{gathered}
\end{equation}
Let $1_\mu$ be the identity of $(01)\in\magma(2)$, and define the five coefficient-one morphisms
\begin{equation}\label{eq:cabledA}
\begin{aligned}
\bA^{0,1,2}&=1_\mu\circ_0\bA\colon p_0\to p_1,&
\bA^{0,12,3}&=\bA\circ_11_\mu\colon p_1\to p_2,\\
\bA^{1,2,3}&=1_\mu\circ_1\bA\colon p_2\to p_3,&
\bA^{01,2,3}&=\bA\circ_01_\mu\colon p_0\to p_4,\\
\bA^{0,1,23}&=\bA\circ_21_\mu\colon p_4\to p_3.
\end{aligned}
\end{equation}

\begin{lemma}\label{lem:pentagonPaRCD}
Let $\Theta$ be an object-fixing continuous operad endomorphism of $\hPaRCD$ with $\Theta(\bA)=g(T_{01},T_{12})\bA$ for some $g\in\exp(\hff)$. Then $\Theta(\bA^B)=s(g^B)\,\bA^B$ for each of the five labels $B$ in \eqref{eq:cabledA}, and $g$ satisfies \textup{(P)}.
\end{lemma}
This is the cabling of the associator \cite{BN}, \cite[Section 10.2]{FresseBook1}: the coefficients follow from \eqref{eq:cabling}, and applying $\Theta$ to Mac Lane's pentagon $\bA^{0,1,2}\bA^{0,12,3}\bA^{1,2,3}=\bA^{01,2,3}\bA^{0,1,23}$ gives (P) in the left-to-right convention. For $g,h\in\GRT_1$, one has 
\begin{equation}\label{eq:grouplaw}
   \Theta_g\circ\Theta_h=\Theta_{g\circledast h} \quad \text{with} \quad  g\circledast h=h(x,gyg^{-1})\,g=g\,h(g^{-1}xg,y). 
\end{equation}
Indeed, $\Theta_g$ fixes $x=T_{01}$ at $((01)2)$, which is $1_\mu\circ_0\bH$, and $T_{12}$ at $(0(12))$,
which is $1_\mu\circ_1\bH$. Since $y=T_{12}$ at $((01)2)$ is the conjugate of the latter by $\bA$,
$\Theta_g(y)=gyg^{-1}$, from which comes $\Theta_g\Theta_h(\bA)=h(x,gyg^{-1})\,g\,\bA$ (compare \cite[(2.2)]{FHK}).
Hence $g\mapsto\Theta_g$ is a group isomorphism $(\GRT_1,\circledast)\cong\Aut^+_{\Op,1}(\hPaRCD)$.

% In particular $g\mapsto\Theta_g$ is a group isomorphism $(\GRT_1,\circledast)\cong\Aut^+_{\Op,1}(\hPaRCD)$. Indeed, $\Theta_g$ fixes $x=T_{01}$ in $((01)2)$ and $T_{12}$ in $(0(12))$, so it acts in $((01)2)$ by $x\mapsto x$, $y=T_{12}\mapsto gyg^{-1}$, from which comes $\Theta_g\Theta_h(\bA)=h(x,gyg^{-1})\,g\,\bA$ (compare \cite[(2.2)]{FHK}). 
% We note that for every \(g\in\GRT_1\), every \(n\), and every \(\sigma\in\Sigma_{n+1}\), $\Theta_{g,n}P_\sigma=P_\sigma\Theta_{g,n}$. Together with Theorem~\ref{thm:classification}, this implies 
% \[
% \Aut^+_{\cyc,1}(\hPaRCD)
% =
% \Aut^+_{\Op,1}(\hPaRCD)
% \cong\GRT_1.
% \]

We show in Section~\ref{sec:cyclic} that every $\Theta_g$ commutes with the cyclic action (Proposition~\ref{prop:cyclicity}).

\section{The flip as a framed reflection}\label{sec:A}

In this section, we prove Theorem~\ref{thmA}. On unframed elements the reflection $P_{\rho_n}$ agrees with the flip up to framings (Proposition~\ref{prop:reflection}), and we compute these framings explicitly (Proposition~\ref{prop:taylor}). The dagger subalgebra is the set where no framings appear, and the set that stays unframed under all of $\Sigma_{n+1}$ is $s\,\UU(\fd_n)$ (Section~\ref{subsec:stable-locus}).

\subsection{The framing defect}
The dagger subalgebra and its $\Sigma_{n+1}$-stable analogue are common kernels of derivations. We first prove the kernel lemma (Lemma~\ref{lem:kernel}) that covers both cases.

\begin{lemma}\label{lem:kernel}
Let $\fg$ be the weight completion of a graded Lie algebra generated in weight one, and let $\fk\subset\fg$ be a closed ideal with $\calF^2\fg\subset\fk$. Let $v_1,\dots,v_r\in\fg$ be of weight one with classes forming a basis of $\fg/\fk$, and let $e_1,\dots,e_r\colon\UU(\fg)\to\kk$ be the continuous linear forms which vanish outside weight one and satisfy $e_i(\fk)=0$ and $e_i(v_j)=\delta_{ij}$. Then $\partial_i=(e_i\otimes\id)\circ\Delta$ are commuting continuous derivations of $\UU(\fg)$, with $\partial_i(w)=e_i(w)$ for $w\in\fg$. Every $u\in\UU(\fg)$ has a unique expansion
\[
u=\sum_{\nu\in\N^r}v^\nu u_\nu,\qquad v^\nu=v_1^{\nu_1}\cdots v_r^{\nu_r},\quad u_\nu\in\UU(\fk),
\]
and $\partial_iu=\sum_\nu\nu_i\,v^{\nu-\varepsilon_i}u_\nu$, where $\varepsilon_i$ is the $i$-th unit vector. In particular
\[
\bigcap_{i=1}^r\ker\partial_i=\UU(\fk).
\]
\end{lemma}
\begin{proof}
The $e_i$ are infinitesimal characters, so the $\partial_i$ are derivations, commuting by cocommutativity. By the Poincar\'e--Birkhoff--Witt theorem for $\fg=\spanlin\{v_i\}\oplus\fk$, and since $\partial_i(\fk)=0$ and $\partial_i(v_j)=\delta_{ij}$, the Leibniz rule gives the formula for $\partial_iu$. If $\partial_iu = 0$, for every $i$, characteristic zero forces $u_\nu=0$ for $\nu\ne0$.
\end{proof}

For $\fg=\ft_n$, $\fk=\fh_n:=\fd_n\oplus\spanlin\{t_{ij}:0<i<j<n\}$ and $v_j=t_{0j}$, the forms $e_j$ are the $e_{0j}$ of Section~\ref{subsec:flip}, and $\fh_n$ is a closed ideal since it is the common kernel of the Lie characters $e_{0j}|_{\ft_n}$.

% \subsection{The reflection}

\begin{prop}\label{prop:reflection}
Let $n\ge2$.
\begin{enumerate}[label=\textup{(\arabic*)}]
\item For $0<i,j<n$,
\begin{equation}\label{eq:defect}
P_{\rho_n}s(t_{ij})=s F_n(t_{ij}),\qquad P_{\rho_n}s(t_{0j})=sF_n(t_{0j})-T_{n-j,\,n-j},
\end{equation}
where $F_n$ is given by \eqref{eq:intro-flip}.
\item The map $q\,P_{\rho_n}\,s$ is the continuous Lie algebra endomorphism of $\ft_n$ given by \eqref{eq:intro-flip}. Thus $F_n=q\,P_{\rho_n}\,s$ is well defined. It is an involutive automorphism of $\ft_n$, and $F_n(c_n)=-c_n$.
\item $P_{\rho_n}s=sF_n$ on $\fd_n$ and on $\UU(\fd_n)$.
\end{enumerate}
\end{prop}

\begin{proof}
(1) For $0<i,j<n$, \eqref{eq:tau} gives $P_{\rho_n}s(t_{ij})=T_{n-i,n-j}$ with $0<n-i,n-j<n$. For $0<j<n$,
\[
P_{\rho_n}s(t_{0j})=T_{n,\,n-j}=-\sum_{\substack{k<n\\k\ne n-j}}T_{k,\,n-j}-T_{n-j,\,n-j},
\]
and the first sum is $sF_n(t_{0j})$.

(2) By \eqref{eq:framing-splitting}, $q$ is a morphism of Lie algebras, so $q P_{\rho_n}s$ is a continuous Lie algebra endomorphism of $\ft_n$. By (1) it is given on generators by \eqref{eq:intro-flip}. Write $F_n=qP_{\rho_n}s$. Again by \eqref{eq:framing-splitting},
\[
P_{\rho_n}s=sF_n+\zeta,
\]
where $\zeta\colon\ft_n\to\calZ_n$ is a continuous Lie algebra map to an abelian Lie algebra, hence vanishes on $\fd_n$. By (1), $\zeta$ takes values in $\calZ_n'=\spanlin\{T_{ii}:0<i<n\}$, and $P_{\rho_n}(\calZ_n')=\calZ_n'\subset\ker q$ by \eqref{eq:tau}. Since $P_{\rho_n}^2=1$ and $qs=\id$,
\[
u=qP_{\rho_n}^2s(u)=qP_{\rho_n}\bigl(sF_n(u)+\zeta(u)\bigr)=F_n^2(u)\qquad(u\in\ft_n),
\]
so $F_n$ is an involutive automorphism. Finally, with $c'=\sum_{0<a<b<n}t_{ab}$ we have $F_n(c')=c'$ and
\[
F_n\Bigl(\sum_{0<m<n}t_{0m}\Bigr)=-\sum_{0<m<n}\ \sum_{k\ne m}t_{km}=-\sum_{0<m<n}t_{0m}-2c',
\]
so $F_n(c_n)=-c_n$.

(3) Since $\zeta$ vanishes on $\fd_n$, $P_{\rho_n}s=sF_n$ on $\fd_n$, and both sides are algebra maps on $\UU(\fd_n)$.
\end{proof}

% Proposition~\ref{prop:reflection}(2) recovers the involutivity of the flip \cite[Lemma~60]{FHK} as a consequence of $\rho_n^2=1$. The reason is that the framing defect \eqref{eq:defect} never involves $T_{00}$, the only framing which $P_{\rho_n}$ moves out of $\ker q$.

% \subsection{The framing defect}

\begin{prop}\label{prop:taylor}
Let $n\ge2$. Identify
\[
\UU(\fft_n)=s\,\UU(\ft_n)\,\widehat\otimes\,\kk[[T_{00},\dots,T_{n-1,n-1}]]
\]
via \eqref{eq:framing-splitting}, and consider the derivation
\[
D=\sum_{j=1}^{n-1}\partial_{0j}\otimes T_{n-j,n-j}.
\]
Then for every $u\in\UU(\ft_n)$
\begin{equation}\label{eq:taylor}
P_{\rho_n}s(u)=(sF_n\otimes\id)\bigl(\exp(-D)(u\otimes1)\bigr)=\sum_{\nu\in\N^{n-1}}\frac{(-1)^{|\nu|}}{\nu!}\,T^\nu\,sF_n\bigl(\partial^\nu u\bigr),
\end{equation}
where $T^\nu=\prod_{j=1}^{n-1}T_{n-j,n-j}^{\nu_j}$, $\partial^\nu=\prod_{j=1}^{n-1}\partial_{0j}^{\nu_j}$ and $\nu!=\prod_j\nu_j!$.
\end{prop}

\begin{proof}
Since the $T_{ii}$ are central, $D$ is a continuous derivation of
\[
\UU(\ft_n)\,\widehat\otimes\,\kk[[T_{00},\dots,T_{n-1,n-1}]].
\]
It lowers the weight in $\UU(\ft_n)$ by one and raises the polynomial degree in the framings by one, so $\exp(-D)$ converges, and it is an algebra automorphism by the Leibniz formula for $D^k(uv)$. The map $sF_n\otimes\id$ is an algebra map because the framings are central. Both sides of the first equality in \eqref{eq:taylor} are therefore continuous algebra maps in $u$, and it suffices to compare them on generators. For $0<a<b<n$, $D(t_{ab})=0$ and both sides equal $sF_n(t_{ab})$. For $0<j<n$, $D(t_{0j}\otimes1)=1\otimes T_{n-j,n-j}$ and $D^2(t_{0j}\otimes 1)=0$, so the right-hand side is $sF_n(t_{0j})-T_{n-j,n-j}$, which is $P_{\rho_n}s(t_{0j})$ by \eqref{eq:defect}. The second equality follows by expanding $\exp(-D)$, the $\partial_{0j}$ commuting.
\end{proof}

\begin{cor}\label{cor:dagger}
Let $n\ge2$ and $u\in\UU(\ft_n)$. The following are equivalent:
\begin{enumerate}[label=\textup{(\roman*)}]
\item $P_{\rho_n}s(u)\in s\,\UU(\ft_n)$.
\item $u\in\UU(\ft_n)^\dagger$.
\item $P_{\rho_n}s(u)=sF_n(u)$.
\end{enumerate}
Moreover $\UU(\ft_n)^\dagger=\UU(\fh_n)$, and $F_n(\fh_n)=\fh_n$.
\end{cor}

\begin{proof}
The monomials $T^\nu$ are linearly independent over $s\,\UU(\ft_n)$, and $\nu\mapsto T^\nu$ is injective. In \eqref{eq:taylor} the coefficient of $T_{n-j,n-j}$ is $-sF_n(\partial_{0j}u)$. If (i) holds, these coefficients vanish, and since $sF_n$ is injective, $\partial_{0j}u=0$ for all $j$, which is (ii). If (ii) holds, every term with $\nu\ne0$ in \eqref{eq:taylor} vanishes, which is (iii). Clearly (iii) implies (i). The equality $\UU(\ft_n)^\dagger=\UU(\fh_n)$ is Lemma~\ref{lem:kernel} with $\fk=\fh_n$ and $v_j=t_{0j}$. Finally $F_n(\fd_n)=\fd_n$ because $F_n$ is an automorphism, and $F_n(t_{ij})=t_{n-i,n-j}$ for $0<i,j<n$.
\end{proof}

\subsection{The \texorpdfstring{$\Sigma_{n+1}$}{Sigma(n+1)}-stable locus}\label{subsec:stable-locus}

\begin{prop}\label{prop:core}
Let $n\ge2$. Then
\[
\bigcap_{\sigma\in\Sigma_{n+1}}P_\sigma\bigl(s\,\UU(\ft_n)\bigr)=s\,\UU(\fd_n),\qquad
\bigcap_{\sigma\in\Sigma_{n+1}}P_\sigma(s\ft_n)=s\fd_n .
\]
\end{prop}

\begin{proof}
Since the framings are central, \eqref{eq:framing-splitting} gives $\cl{[\fft_n,\fft_n]}=s\fd_n$. Lie algebra automorphisms preserve closed derived subalgebras, so $P_\sigma(s\fd_n)=s\fd_n$ and $P_\sigma(s\,\UU(\fd_n))=s\,\UU(\fd_n)\subset s\,\UU(\ft_n)$ for all $\sigma$. This proves the inclusions $\supset$.

Conversely, let $s(u)$ lie in the left-hand intersection, so that $P_\sigma s(u)\in s\,\UU(\ft_n)$ for all $\sigma\in\Sigma_{n+1}$. Fix $0\le i<n$ and $\pi\in\Sigma_n$ with $\pi(i)=0$. Then $P_{\rho_n}s(P_\pi u)=P_{\rho_n}P_\pi s(u)\in s\,\UU(\ft_n)$, so $P_\pi u\in\UU(\ft_n)^\dagger$ by Corollary~\ref{cor:dagger}. For $a\ne b$ let $e_{ab}$ extract the weight-one coefficient of $t_{ab}$ and $\partial_{ab}=(e_{ab}\otimes\id)\circ\Delta$. Then $P_\pi\partial_{ab}P_\pi^{-1}=\partial_{\pi(a)\pi(b)}$. Hence $\partial_{0k}(P_\pi u)=0$ for $0<k<n$ gives $\partial_{ik}u=0$ for all $k\ne i$. As $i$ is arbitrary, $\partial_{ab}u=0$ for all $a<b$, and Lemma~\ref{lem:kernel} with $\fk=\fd_n$ and $\{v\}=\{t_{ab}\}_{a<b}$ gives $u\in\UU(\fd_n)$.

For the Lie statement, $P_\sigma$ and $s$ are morphisms of Hopf algebras, so $P_\sigma(s\ft_n)=P_\sigma(s\,\UU(\ft_n))\cap\fft_n$, and intersecting the first identity with $\fft_n=\Prim\UU(\fft_n)$ and using $\Prim\UU(\fd_n)=\fd_n$ gives the second.
\end{proof}

\begin{cor}\label{cor:suboperad}
The family $s\fd=(s\fd_n)_{n\ge1}$ is a cyclic suboperad of $\fft$ in complete Lie algebras, and $s\fd_n=\cl{[\fft_n,\fft_n]}$. If $(V_n)_{n\ge1}$ is any family of subspaces $V_n\subset s\ft_n$ with $P_\sigma(V_n)=V_n$ for all $\sigma\in\Sigma_{n+1}$, then $V_n\subset s\fd_n$ for all $n$. In particular, $s\fd$ is the largest cyclic suboperad of $\fft$ contained in $s\ft$.
\end{cor}

\begin{proof}
The equality $s\fd_n=\cl{[\fft_n,\fft_n]}$ was shown in the proof of Proposition~\ref{prop:core}. Each $\circ_a\colon\fft_m\oplus\fft_k\to\fft_{m+k-1}$ and each $P_\sigma$ is a continuous Lie algebra map, and the closed derived algebra of $\fft_m\oplus\fft_k$ is $s\fd_m\oplus s\fd_k$. Hence $s\fd$ is closed under composition and the cyclic action. If $V_n\subset s\ft_n$ is $\Sigma_{n+1}$-stable, then $V_n=P_\sigma(V_n)\subset P_\sigma(s\ft_n)$ for every $\sigma$, and Proposition~\ref{prop:core} gives $V_n\subset s\fd_n$.
\end{proof}

\begin{remark}\label{rem:noquotient}
The framings cannot be discarded operadically: $qP_{\rho_n}(T_{00})=2c_n\ne0=F_nq(T_{00})$, and $q$ is not a morphism of operads by \eqref{eq:cabling}. The flip is Ihara's action of $\rho_n$ on $\ft_n/\kk c_n$ \cite[\S5.3]{ihara_braids}, but $(\kk c_n)_n$ is not an operadic ideal, since $c_m\circ_a0=c_{m+k-1}-c_B$.
\end{remark}

\begin{proof}[Proof of Theorem~\ref{thmA}]
The identity $F_n=qP_{\rho_n}s$ is Proposition~\ref{prop:reflection}(2), and $P_{\rho_n}s=sF_n$ on $\UU(\fd_n)$ is Proposition~\ref{prop:reflection}(3). Part (2) is Proposition~\ref{prop:taylor}. The description of $\UU(\ft_n)^\dagger$ and the identity $P_{\rho_n}s=sF_n$ on it are Corollary~\ref{cor:dagger}. The intersection formula is Proposition~\ref{prop:core}, and the statement on $s\fd$ is Corollary~\ref{cor:suboperad}.
\end{proof}

\section{The flip equation and the reflected pentagon}\label{sec:arity4}
We now set $n=4$ and prove Theorem~\ref{thmB}(1). The flip sends one path around the pentagon to the other, so the pentagon becomes the flip equation, and applying the flip twice gives the two-cycle relation. In Section~\ref{sec:pentagon} we interpret these statements as a mirror symmetry of the pentagon in $\hPaRCD(4)$.

\subsection{The flip equation}

We repeatedly use the shift property \cite[proof of Lemma~5]{Furusho_pentagon},  and \cite[Lemma~77]{FHK}: if $\psi\in\hff'$ or $\psi\in\exp(\hff')$ and $w$ commutes with $u$ and $v$, then
\begin{equation}\label{eq:shift}
\psi(u+w,v)=\psi(u,v+w)=\psi(u,v),
\end{equation}
since every Lie word occurring in $\psi$ has length at least two.

Let $\pi_0\colon\ft_4\to\ft_{\{1,2,3\}}$ delete the strand $0$, that is, $\pi_0(t_{0j})=0$ and $\pi_0(t_{ij})=t_{ij}$ for $i,j>0$, and let $\iota\colon\ft_{\{1,2,3\}}\hookrightarrow\ft_4$ be the inclusion, so that $\pi_0\iota=\id$. Put $\fd_{123}=\iota\,\cl{[\ft_{\{1,2,3\}},\ft_{\{1,2,3\}}]}\subset\fd_4$. Since $\ft_{\{1,2,3\}}$ is the direct sum of the free Lie algebra on $t_{12},t_{23}$ and the centre $\kk c_{123}$ \cite{Drin}, the substitution $x\mapsto t_{12}$, $y\mapsto t_{23}$ is injective on $\UU(\hff)$ and identifies $\hff'$ with $\fd_{123}$ and $\exp(\hff')$ with $\exp(\fd_{123})$.

\begin{prop}\label{prop:pathflip}
For $\psi\in\hff'$,
\[
F_4(\psi^{0,1,2})=\psi^{01,2,3},\qquad F_4(\psi^{0,12,3})=\psi^{0,1,23},\qquad F_4(\psi^{1,2,3})=\psi(t_{23},t_{12}),
\]
and the same identities hold for $g\in\exp(\hff')$. Consequently $F_4(S_\psi)=S'_\psi$ and $F_4(S_g)=S'_g$.
\end{prop}

\begin{proof}
Put $c_4=\sum_{i<j<4}t_{ij}$. By \eqref{eq:intro-flip}, $F_4(t_{01})=-(t_{03}+t_{13}+t_{23})=(t_{02}+t_{12})+(t_{01}-c_4)$, $F_4(t_{12})=t_{23}$,
$F_4(t_{01}+t_{02})=-(t_{02}+t_{03}+t_{12}+t_{13}+2t_{23})=t_{01}-(c_4+t_{23})$, and $F_4(t_{13}+t_{23}) =t_{12}+t_{13}$.

% \[
% \begin{aligned}
% F_4(t_{01})&=-(t_{03}+t_{13}+t_{23})=(t_{02}+t_{12})+(t_{01}-c_4), &
% F_4(t_{12})&=t_{23},\\
% F_4(t_{01}+t_{02})&=-(t_{02}+t_{03}+t_{12}+t_{13}+2t_{23})=t_{01}-(c_4+t_{23}), &
% F_4(t_{13}+t_{23})&=t_{12}+t_{13}.
% \end{aligned}
% \]
The element $t_{01}-c_4$ commutes with $t_{02}+t_{12}$, by the four-term relation $[t_{01},t_{02}+t_{12}]=0$, and with $t_{23}$. The element $c_4+t_{23}$ commutes with $t_{01}$ and, by $[t_{23},t_{12}+t_{13}]=0$, with $t_{12}+t_{13}$. By \eqref{eq:shift}, we get $F_4(\psi^{0,1,2})=\psi(t_{02}+t_{12},t_{23})=\psi^{01,2,3}$ and $F_4(\psi^{0,12,3})=\psi(t_{01},t_{12}+t_{13})=\psi^{0,1,23}$. The third identity follows from $F_4(t_{12})=t_{23}$ and $F_4(t_{23})=t_{12}$. The group-like case is identical.
\end{proof}

The group-like identities appear in the proof of \cite[Theorem~78(1)]{FHK}, written as $F_4(g^{0,1,2}g^{0,12,3})=g^{4,3,2}g^{4,32,1}$.

% and the equivalence of (a), (b) and (c) below is the case $\alpha=0$ of that proof.

\begin{prop}[The flip equation]\label{prop:flipequation}
For $g\in\exp(\hff')$ put $C_g=S_g^{-1}F_4(S_g)$. The following are equivalent:
\begin{enumerate}[label=\textup{(\alph*)}]
\item $C_g\in\exp(\fd_{123})$.
\item $C_g=g^{1,2,3}$.
\item $g$ satisfies \textup{(P)}.
\item $P_{\rho_4}s(S_g)=s\bigl(S_g\,g^{1,2,3}\bigr)$ in $\UU(\fft_4)$.
\end{enumerate}
Moreover $P_{\rho_4}s(S_g)=s(S'_g)$ for every $g\in\exp(\hff')$. For $\psi\in\hff'$ the following are equivalent: \textup{(a$'$)} $(F_4-1)S_\psi\in\fd_{123}$; \textup{(b$'$)} $(F_4-1)S_\psi=\psi^{1,2,3}$; \textup{(c$'$)} $\psi$ satisfies \textup{(P$'$)}; \textup{(d$'$)} $P_{\rho_4}s(S_\psi)=s(S_\psi+\psi^{1,2,3})$.
\end{prop}

\begin{proof}
Each substitution $g^{A,B,C}$ lies in $\exp(\fd_4)$ since $g\in\exp(\hff')$, so $S_g\in\UU(\fd_4)$, and Proposition~\ref{prop:reflection}(3) and Proposition~\ref{prop:pathflip} give $P_{\rho_4}s(S_g)=sF_4(S_g)=s(S'_g)$. Hence $C_g=S_g^{-1}S'_g$, which gives (b)$\Leftrightarrow$(c), and (c)$\Leftrightarrow$(d) because $s$ is injective. Clearly (b)$\Rightarrow$(a). For (a)$\Rightarrow$(b), note that $g(0,v)=1$ for $g\in\exp(\hff')$, so
\[
\begin{gathered}
\pi_0(S_g)=g(0,t_{12})\,g(0,t_{13}+t_{23})=1,\\
\pi_0(S'_g)=g(t_{12},t_{23})\,g(0,t_{12}+t_{13})=g(t_{12},t_{23}),
\end{gathered}
\]
hence $\pi_0(C_g)=\pi_0(g^{1,2,3})$. Both $C_g$ and $g^{1,2,3}$ lie in $\iota\,\UU(\ft_{\{1,2,3\}})$, on which $\pi_0$ is injective, so $C_g=g^{1,2,3}$. The Lie case is identical, with $\pi_0(S_\psi)=0$ and $\pi_0(S'_\psi)=\psi(t_{12},t_{23})$.
\end{proof}

% The equivalences (a)$\Leftrightarrow$(b)$\Leftrightarrow$(c) are contained in the proof of \cite[Theorem~78]{FHK}, where it is shown that the correction series $h$ of Section~\ref{subsec:flip} satisfies $g=h\exp(\alpha y)$ for some $\alpha\in\kk$, and for $g\in\exp(\hff')$ this forces $h=g$. Condition (d) is the framed form of the flip equation.

\begin{cor}\label{cor:fliplocus} We have
\[
\begin{gathered}
\bigl\{g\in\exp(\calF^3\hff):C_g\in\exp(\fd_{123})\bigr\}=\GRT_1,\\
\bigl\{\psi\in\calF^3\hff:(F_4-1)S_\psi\in\fd_{123}\bigr\}=\lgrt_1 .
\end{gathered}
\]
Without the normalisation, $\{\psi\in\hff':(F_4-1)S_\psi\in\fd_{123}\}=\kk\cdot [x,y]\oplus\lgrt_1$, and $\{\psi\in\kk y\oplus\hff':(F_4-1)S_\psi\in\fd_{123}\}=\kk y\oplus\kk\cdot [x,y]\oplus\lgrt_1$.
\end{cor}

\begin{proof}
The first two identities follow from Proposition~\ref{prop:flipequation} and Definition~\ref{def:grt}, and the third from Remark~\ref{rem:normalisation}. For the last, $S_y=t_{12}+t_{13}+t_{23}$ is fixed by $F_4$, so $y$ satisfies the flip condition, and the condition is linear in $\psi$.
\end{proof}

\begin{prop}\label{prop:astr}
Let $g\in\exp(\hff)$. Then $P_{\rho_4}s(S_g)\in s\,\UU(\ft_4)$ if and only if the coefficient of $x$ in $g$ vanishes. In that case, $P_{\rho_4}s(S_g)=s\bigl(S_g\,h^{1,2,3}\bigr)$ for some $h\in\exp(\hff')$ if and only if $g=\varphi\exp(\alpha y)$ with $\varphi\in\ASTR$ and $\alpha\in\kk$, and then $h=\varphi$.
\end{prop}

\begin{proof}
For group-like $u$, $\Delta u=u\otimes u$ gives $\partial_{0j}u=e_{0j}(\log u)\,u$. If $\beta$ and $\alpha$ are the coefficients of $x$ and $y$ in $g$, the weight-one part of $\log S_g$ is $\beta(2t_{01}+t_{02})+\alpha(t_{12}+t_{13}+t_{23})$, so $S_g\in\UU(\ft_4)^\dagger$ if and only if $\beta=0$. The first assertion is Corollary~\ref{cor:dagger}, (i)$\Leftrightarrow$(ii). If $\beta=0$, Corollary~\ref{cor:dagger} gives $P_{\rho_4}s(S_g)=sF_4(S_g)$, and since $s$ is injective the second assertion is \cite[Theorem~78(1)]{FHK}, whose proof holds over any field of characteristic zero.
\end{proof}

\begin{cor}\label{cor:twocycle}
Let $g\in\exp(\hff)$ satisfy \textup{(P)}. Then $g\in\exp(\hff')$ and $g(x,y)\,g(y,x)=1$. Similarly, if $\psi\in\hff'$ satisfies \textup{(P$'$)}, then $\psi(x,y)+\psi(y,x)=0$.
\end{cor}

\begin{proof}
By \cite[Lemma~75(1)]{FHK}, $g\in\exp(\hff')$. By Proposition~\ref{prop:flipequation}, $F_4(S_g)=S_g\,g^{1,2,3}$. Apply $F_4$ and use $F_4^2=\id$, which on $\UU(\fd_4)$ is the relation $P_{\rho_4}^2=1$ read through $P_{\rho_4}s=sF_4$, together with $F_4(g^{1,2,3})=g(t_{23},t_{12})$:
\[
S_g=F_4\bigl(S_g\,g^{1,2,3}\bigr)=S_g\,g^{1,2,3}\,g(t_{23},t_{12}).
\]
Hence $g(t_{12},t_{23})\,g(t_{23},t_{12})=1$, and injectivity of $x\mapsto t_{12}$, $y\mapsto t_{23}$ on $\UU(\hff)$ gives $g(x,y)g(y,x)=1$. The Lie case is identical.
\end{proof}

Corollary~\ref{cor:twocycle} should be compared with \cite[Lemma~6]{Furusho_pentagon}, where the two-cycle relation is obtained by projecting the pentagon onto a free Lie algebra. Since $\ASTR$ is exactly the set of group-like solutions of \textup{(P)}, Corollary~\ref{cor:twocycle} holds for every associator, and Propositions~\ref{prop:reflectedpentagon} and ~\ref{prop:arity3} apply to every $\varphi\in\ASTR$. 

\begin{proof}[Proof of Theorem~\ref{thmB}(1)] The first assertion is Proposition~\ref{prop:astr}. If $g$ has no linear term in $x$, write $g=\varphi\exp(\alpha y)$ with $\varphi\in\exp(\hff')$. Then $g^{A,B,C}=\varphi^{A,B,C}\exp(\alpha t_{BC})$, so Proposition~\ref{prop:pathflip}, $F_4(t_{12})=t_{23}$ and $F_4(t_{13}+t_{23})=t_{12}+t_{13}$ give $F_4(S_g)=S'_g$, and Corollary~\ref{cor:dagger}, (i)$\Rightarrow$(iii), gives $P_{\rho_4}s(S_g)=sF_4(S_g)$. The identification with the flip equation is Proposition~\ref{prop:astr}, and the two-cycle relation is Corollary~\ref{cor:twocycle}.
\end{proof}

\subsection{The reflected pentagon}\label{sec:pentagon}

We now read Proposition~\ref{prop:flipequation} in the groupoid $\hPaRCD(4)$. Let $M$ be the involution of $\magma(4)$ which exchanges the two children at every vertex of a tree, and define
\[
\kappa_4(p)=M\bigl(\rho_4(p)\bigr),\qquad \kappa_4\bigl(u\colon p\to q\bigr)=\bigl(P_{\rho_4}(u)\colon\kappa_4(p)\to\kappa_4(q)\bigr),
\]
where $\rho_4$ acts on objects through the cyclic structure of $\magma$. Since hom-sets are constant and $P_{\rho_4}$ is an algebra automorphism, $\kappa_4$ is an automorphism of the groupoid $\hPaRCD(4)$. It is involutive because $M$ commutes with relabelling and $M^2=\id=\rho_4^2$. The mirror $M$ is not part of the cyclic structure. It only restores the leaf order $0123$, which $\rho_4$ reverses. By Proposition~\ref{prop:reflection}(3), $\kappa_4$ acts on coefficients in $s\,\UU(\fd_4)$ by $F_4$.

For $g\in\exp(\hff')$ let $\bA_g^B$ be the morphism with the same source and target as $\bA^B$ in \eqref{eq:cabledA} and with coefficient $s(g^B)$, and put $\alpha_g=\bA_g^{0,1,2}\bA_g^{0,12,3}\colon p_0\to p_2$. By Lemma~\ref{lem:pentagonPaRCD}, if $g\in\GRT_1$ then $\bA_g^B=\Theta_g(\bA^B)$.

\begin{prop}\label{prop:reflectedpentagon}
Let $g\in\exp(\hff')$. The involution $\kappa_4$ fixes $p_0$, exchanges $p_1\leftrightarrow p_4$ and $p_2\leftrightarrow p_3$, and satisfies $\kappa_4(\bA_g^{0,1,2})=\bA_g^{01,2,3}$ and $\kappa_4(\bA_g^{0,12,3})=\bA_g^{0,1,23}$. Hence $\alpha_g^{-1}\kappa_4(\alpha_g)\colon p_2\to p_3$ has coefficient $s(C_g)$, and $\alpha_g^{-1}\kappa_4(\alpha_g)=\bA_g^{1,2,3}$ if and only if $g$ satisfies \textup{(P)}. In that case $\kappa_4(\bA_g^{1,2,3})=(\bA_g^{1,2,3})^{-1}$, which is the two-cycle relation.
\end{prop}

\begin{proof}
As unrooted trees with boundary order $(0,1,2,3,4)$, the $p_i$ are determined by the two-element parts of their splits, $p_0:\{01,34\}$, $p_1:\{12,34\}$, $p_2:\{12,40\}$, $p_3:\{23,40\}$, $p_4:\{01,23\}$, and $\rho_4$ relabels these by $01\leftrightarrow34$, $12\leftrightarrow23$, and $M$ restores the boundary order. On coefficients $\kappa_4$ acts on $s\,\UU(\fd_4)$ by $F_4$, so the remaining assertions are Propositions~\ref{prop:pathflip}, \ref{prop:flipequation} and Corollary~\ref{cor:twocycle}.
\end{proof}

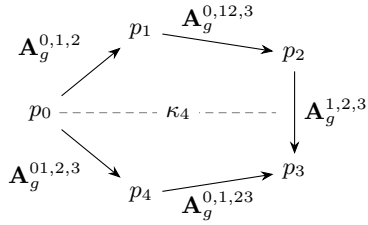
\begin{figure}[htbp]
\centering
\begin{tikzpicture}[>=Stealth,scale=0.7,every node/.style={font=\small}]
\node (p0) at (-2.6,0) {$p_0$};
\node (p1) at (-0.7,1.55) {$p_1$};
\node (p2) at (2.2,1.1) {$p_2$};
\node (p3) at (2.2,-1.1) {$p_3$};
\node (p4) at (-0.7,-1.55) {$p_4$};
\draw[->] (p0) -- node[above left] {$\bA_g^{0,1,2}$} (p1);
\draw[->] (p1) -- node[above] {$\bA_g^{0,12,3}$} (p2);
\draw[->] (p2) -- node[right] {$\bA_g^{1,2,3}$} (p3);
\draw[->] (p0) -- node[below left] {$\bA_g^{01,2,3}$} (p4);
\draw[->] (p4) -- node[below] {$\bA_g^{0,1,23}$} (p3);
\draw[densely dashed,gray] (-2.25,0) -- (1.85,0);
\node[fill=white,inner sep=3pt] at (0,0) {$\kappa_4$};
\end{tikzpicture}
\caption{The involution $\kappa_4$ fixes $p_0$, swaps the two paths from $p_0$, and reverses the orientation of the edge between $p_2$ and $p_3$.}
\label{fig:pentagon}
\end{figure}

\section{Cyclic compatibility of the \texorpdfstring{$\GRT_1$}{GRT1}-action}\label{sec:cyclic}

We first give a criterion for an operad automorphism of $\hPaRCD$ to be cyclic.
% : it suffices to check the generators, and on each generator only the transpositions $\tau_m$ (Lemma~\ref{lem:generators}). 
We then show that the two-cycle relation is exactly the condition for $\Theta_g$ to commute with the cyclic action on the associativity morphism (Proposition~\ref{prop:arity3}). Together with the generator criterion, this shows that $\GRT_1$ acts by cyclic automorphisms of $\hPaRCD$ and completes the proof of Theorem~\ref{thmB}.

\subsection{A criterion for cyclicity}\label{subsec:criterion} 
For a cyclic operad $\mathcal P$, the partial compositions are compatible with the action of $\Sigma_{n+1}$ in the following sense \cite{gk_cyc}, \cite{mms_operads_topology_and_phys}. A composite $x\circ_ay$, with $x\in\mathcal P(m)$ and $y\in\mathcal P(k)$, is a two-vertex tree with $m+k$ labelled flags. Relabelling these flags by $\sigma\in\Sigma_{m+k}$ and re-rooting gives again a two-vertex tree. Hence there exist $b$, $\sigma'\in\Sigma_{m+1}$, $\sigma''\in\Sigma_{k+1}$ and $\pi\in\Sigma_{m+k-1}$, depending only on $(m,k,a,\sigma)$, such that for all $x,y$
\begin{equation}\label{eq:cyclic-equivariance}
P_\sigma(x\circ_ay)=P_\pi\bigl(P_{\sigma'}x\circ_bP_{\sigma''}y\bigr)\quad\text{or}\quad P_\sigma(x\circ_ay)=P_\pi\bigl(P_{\sigma''}y\circ_bP_{\sigma'}x\bigr).
\end{equation}

\begin{lemma}\label{lem:generators}
Let $\Theta\in\Aut^+_{\Op}(\hPaRCD)$. Suppose that $\Theta P_\sigma(z)=P_\sigma\Theta(z)$ for every $z\in\{\exp(\lambda\bI),\exp(\lambda\bH),\bX,\bA:\lambda\in\kk\}$ of arity $m$ and every $\sigma\in\Sigma_{m+1}$. Then $\Theta$ commutes with the whole cyclic action. Moreover, for a morphism $z$ of arity $m$, the condition $\Theta P_\sigma(z)=P_\sigma\Theta(z)$ for all $\sigma\in\Sigma_{m+1}$ holds if and only if $\Theta P_{\tau_m}(P_\pi z)=P_{\tau_m}\Theta(P_\pi z)$ for all $\pi\in\Sigma_m$.
\end{lemma}

\begin{proof}
For the second assertion, $\Sigma_m$ is the stabiliser of $m$ in $\Sigma_{m+1}$ and acts transitively on $\{0,\dots,m-1\}$, so $\Sigma_{m+1}=\Sigma_m\sqcup\Sigma_m\tau_m\Sigma_m$. As $\Theta$ commutes with $P_\pi$ for $\pi\in\Sigma_m$, for $\sigma=\pi_1\tau_m\pi_2$ we get
\[
\Theta P_\sigma z=P_{\pi_1}\Theta P_{\tau_m}(P_{\pi_2}z)=P_{\pi_1}P_{\tau_m}\Theta(P_{\pi_2}z)=P_\sigma\Theta z
\]
whenever $\Theta$ commutes with $P_{\tau_m}$ on $P_{\pi_2}z$. The converse is clear.

For the first assertion, let $\mathcal M$ be the set of morphisms $z$, of any arity $m$, such that $\Theta P_\sigma z=P_\sigma\Theta z$ for all $\sigma\in\Sigma_{m+1}$. Since $\Theta$ and each $P_\sigma$ are continuous functors, $\mathcal M$ is closed under composition, inverses and limits. If $x,y\in\mathcal M$, then by \eqref{eq:cyclic-equivariance} and since $\Theta$ is an operad morphism,
\begin{align*}
\Theta P_\sigma(x\circ_ay)&=P_\pi\bigl(\Theta P_{\sigma'}x\circ_b\Theta P_{\sigma''}y\bigr)=P_\pi\bigl(P_{\sigma'}\Theta x\circ_bP_{\sigma''}\Theta y\bigr)\\
&=P_\sigma(\Theta x\circ_a\Theta y)=P_\sigma\Theta(x\circ_ay),
\end{align*}
and similarly in the second case of \eqref{eq:cyclic-equivariance}. Thus $\mathcal M$ is closed under operadic composition. It contains the generators by hypothesis, hence all morphisms.
\end{proof}

\subsection{The cyclic condition on the associativity morphism} Let $g\in\exp(\hff')$. We extend the assignment $\bA\mapsto g(T_{01},T_{12})\bA$ $\Sigma_3$-equivariantly: for every permutation $(a,b,c)$ of $(0,1,2)$ we put
\begin{equation}\label{eq:thetaext}
\Theta_g(A_{abc})=g(T_{ab},T_{bc})\,A_{abc},\qquad \Theta_g(A_{abc}^{-1})=\Theta_g(A_{abc})^{-1}.
\end{equation}
By Lemma~\ref{lem:tau-generators} and the argument below, $P_\sigma\bA$ is of the form $A_{abc}^{\pm1}$ for every $\sigma\in\Sigma_4$, so the condition $\Theta_gP_\sigma\bA=P_\sigma\Theta_g\bA$ makes sense.

\begin{lemma}\label{lem:injective}
Let $u,v$ be two distinct elements of $\{T_{01},T_{12},T_{02}\}\subset\fft_3$. The continuous algebra map $\UU(\hff)\to\UU(\fft_3)$, $x\mapsto u$, $y\mapsto v$, is injective. In particular $g(u,v)g(v,u)=1$ if and only if $g(x,y)g(y,x)=1$.
\end{lemma}

\begin{proof}
Let $\calZ\subset\fft_3$ be the central ideal spanned by $T_{00},T_{11},T_{22}$ and $s(c_3)$. By \eqref{eq:framing-splitting} and \cite{Drin}, $\fft_3/\calZ\cong\ft_3/\kk c_3$ is freely generated by the classes of $T_{01}$ and $T_{12}$, and the class of $T_{02}$ is $-T_{01}-T_{12}$. Any two of the three classes form a free generating pair, since $(x,y)\mapsto(-x-y,y)$ and $(x,y)\mapsto(x,-x-y)$ are automorphisms of $\hff$. Hence the composite $\UU(\hff)\to\UU(\fft_3)\to\UU(\fft_3/\calZ)$ is an isomorphism.
\end{proof}

\begin{prop}\label{prop:arity3}
Let $g\in\exp(\hff')$. The following are equivalent:
\begin{enumerate}[label=\textup{(\roman*)}]
\item $\Theta_gP_\sigma\bA=P_\sigma\Theta_g\bA$ for every $\sigma\in\Sigma_4$.
\item $\Theta_gP_{\tau_3}\bA=P_{\tau_3}\Theta_g\bA$.
\item $g(x,y)\,g(y,x)=1$.
\end{enumerate}
\end{prop}

\begin{proof}
Write $\tau=\tau_3$ and $\gamma_{abc}=g(T_{ab},T_{bc})$. Proceeding as in the proof of Lemma~\ref{lem:tau-generators}, one finds
\[
P_\tau A_{012}=A_{120}^{-1},\qquad P_\tau A_{021}=A_{210}^{-1},\qquad P_\tau A_{102}=A_{201}.
\]
For instance, $((10)2)$ has boundary order $(1,0,2,3)$ and split $\{1,0\}\,|\,\{2,3\}$. After relabelling by $(0\,3)$ and re-rooting at $3$ it becomes $((20)1)$, and $(1(02))$ becomes $(2(01))$. Since $P_\tau^2=1$ and $P_\tau$ is a functor, the three remaining cases are $P_\tau A_{120}=A_{012}^{-1}$, $P_\tau A_{210}=A_{021}^{-1}$ and $P_\tau A_{201}=A_{102}$. For $z=A_{abc}$ and $P_\tau z=A_{a'b'c'}^{\epsilon}$ both $\Theta_gP_\tau z$ and $P_\tau\Theta_g z$ have the same source and target, with coefficients $\gamma_{a'b'c'}^{\epsilon}$ and $P_\tau(\gamma_{abc})$, so compatibility on $z$ reads $P_\tau(\gamma_{abc})=\gamma_{a'b'c'}^{\epsilon}$. By \eqref{eq:output},
\[
\begin{gathered}
P_\tau(T_{01})=T_{31}=T_{02}-s(c_3)-T_{11},\quad P_\tau(T_{02})=T_{32}=T_{01}-s(c_3)-T_{22},\\
P_\tau(T_{12})=T_{12},
\end{gathered}
\]
and $s(c_3)+T_{11}$, $s(c_3)+T_{22}$ are central in $\fft_3$. By \eqref{eq:shift}, applied to each argument:
\begin{itemize}
\item for $z=A_{012}$: $P_\tau(\gamma_{012})=g(T_{02},T_{12})$ and $\gamma_{120}^{-1}=g(T_{12},T_{02})^{-1}$,
\item for $z=A_{021}$: $P_\tau(\gamma_{021})=g(T_{32},T_{12})=g(T_{01},T_{12})$ and $\gamma_{210}^{-1}=g(T_{12},T_{01})^{-1}$,
\item for $z=A_{102}$: $P_\tau(\gamma_{102})=g(T_{31},T_{32})=g(T_{02},T_{01})=\gamma_{201}$.
\end{itemize}
The third condition always holds. By Lemma~\ref{lem:injective}, each of the first two is equivalent to (iii). In particular (ii)$\Leftrightarrow$(iii). Compatibility on $z$ is equivalent to compatibility on $P_\tau z$ and on $z^{-1}$. Hence, if (iii) holds, then $\Theta_g$ commutes with $P_\tau$ on every $P_\pi\bA$, $\pi\in\Sigma_3$. The argument for the second assertion of Lemma~\ref{lem:generators} uses only that $\Theta_g$ commutes with $P_\pi$, $\pi\in\Sigma_3$, on the morphisms $A^{\pm1}_{abc}$, which holds by \eqref{eq:thetaext}. Hence (i) holds. Clearly (i)$\Rightarrow$(ii). 
\end{proof}

% Finally, we establish the cyclicity of the $\GRT_1$-action.
\begin{prop}\label{prop:cyclicity}
Let $g\in\GRT_1$. The automorphism $\Theta_g$ of Theorem~\ref{thm:classification} lies in the subgroup $\Aut^+_{\cyc,1}(\hPaRCD)$.
\end{prop}

\begin{proof}
We verify the hypothesis of Lemma~\ref{lem:generators}. The restriction of $\Theta_g$ to $\hPaRCD(1)$ and $\hPaRCD(2)$ is the identity. Indeed, $\hPaRCD(1)$ has the single object $(0)$ with automorphism group $\exp(\kk\bI)$. In arity two, the Lie algebra $\fft_2$ is abelian and spanned by $T_{00}$, $T_{11}$ and $T_{01}$, which at $(01)$ are the insertions $1_\mu\circ_0\bI$, $1_\mu\circ_1\bI$ and $\bH$. These are fixed by $\Theta_g$, hence so is $\Aut((01))=\exp(\fft_2)$, and so are $\bX$ and $\Aut((10))=\bX^{-1}\Aut((01))\bX$. Since each $P_\sigma$ preserves arity, $\Theta_g$ commutes with the cyclic action on $\exp(\kk\bI)$, $\exp(\kk\bH)$ and $\bX$. On $\bA$, $g$ satisfies the two-cycle relation by Corollary~\ref{cor:twocycle}, so Proposition~\ref{prop:arity3} applies. Lemma~\ref{lem:generators} now shows that $\Theta_g$ is cyclic.
\end{proof}

\begin{proof}[Proof of Theorem~\ref{thmB}(2)] This is Proposition~\ref{prop:arity3}. For the final statement, by Proposition~\ref{prop:cyclicity}, $g\mapsto\Theta_g$ maps $\GRT_1$ into $\Aut^+_{\cyc,1}(\hPaRCD)\subset\Aut^+_{\Op,1}(\hPaRCD)$, and by Theorem~\ref{thm:classification} and equation~\eqref{eq:grouplaw} it is a group isomorphism onto $\Aut^+_{\Op,1}(\hPaRCD)$. Hence $\Aut^+_{\cyc,1}(\hPaRCD)=\Aut^+_{\Op,1}(\hPaRCD)$.
\end{proof}

\bibliographystyle{plain}
\bibliography{bib}

\end{document}